\documentclass[runningheads, envcountsame, a4paper]{article}

\usepackage{amssymb}
\usepackage{amsmath}
\usepackage{ amsfonts} 
\usepackage{ theorem}  
\usepackage{times}
\usepackage{stmaryrd}
\usepackage{mathrsfs}
\usepackage{latexsym}
\usepackage{hhline}

\usepackage{color}
\usepackage{graphicx}
\usepackage[skip=0pt]{caption}
\usepackage{subcaption}

\usepackage{tikz}
\usepackage{lineno}
\usepackage{todonotes}
\usepackage{bm}

\theoremstyle{plain}
\newtheorem{theorem}{Theorem}
\newtheorem{corollary}[theorem]{Corollary}
\newtheorem{lemma}[theorem]{Lemma}

\newtheorem{proposition}[theorem]{Proposition}

\theoremstyle{definition}
\newtheorem{definition}[theorem]{Definition}
\newtheorem{example}[theorem]{Example}
\newtheorem{remark}[theorem]{Remark}

\newenvironment{proof}{\noindent{\bf Proof:\/}}{$\Box$\vskip 0.1in}

 \def\dist{{\rm dist}}

\title{ {\Large Hypercubes and Isometric Words\\
{based on}
Swap and {Mismatch} Distance		}
\thanks{ Partially supported
by INdAM-GNCS Project 2022 and 2023, FARB Project  ORSA229894 
of University of Salerno, TEAMS Project of University of Catania and by  the MIUR Excellence Department Project MatMod@TOV awarded to the Department of Mathematics, University of Rome Tor Vergata.
 }
}
\author{Author One$^1$\thanks{Author One was partially supported by Grant XXX} \and Author Two$^2$ \and Author Three$^1$}

\date{
	$^1$Organization 1 \\ \texttt{\{auth1, auth3\}@org1.edu}\\%
	$^2$Organization 2 \\ \texttt{auth3@inst2.edu}\\[2ex]%
}

\author{ M. Anselmo$^1$ \and  G. Castiglione$^2$ \and  M. Flores$^{1,2}$ \and   D. Giammarresi$^3$ \and   M. Madonia$^4$ \and S. Mantaci$^2$
  }%

  \date{	
$^1${ \normalsize Dipartimento di Informatica,
Universit\`a di Salerno, 
Italy. 
\\
{\tt
\{manselmo, mflores \}@unisa.it}} \\
$^2${ \normalsize Dipartimento di Matematica e Informatica, Universit\`a di Palermo, 
Italy  
{\tt \{giuseppa.castiglione, sabrina.mantaci \}@unipa.it }} \\
$^3${ \normalsize Dipartimento di Matematica.
 Universit\`a  Roma ``Tor Vergata''
  Italy. 
{\tt giammarr@mat.uniroma2.it}}\\
$^4${\normalsize Dipartimento di Matematica eInformatica, Universit\`a  di
  Catania, 
  Italy.
  {\tt madonia@dmi.unict.it}} 
}

\begin{document}

\thispagestyle{empty}

 \maketitle

\begin{abstract} 
The hypercube of dimension $n$ is the graph whose vertices are 
the $2^n$ binary words of length $n$, 
and there is an edge between two of them if 
they have Hamming distance $1$.  
We consider an edit distance based on swaps and mismatches,  
to which we refer as {\em tilde-distance},
and define the 
 {\em tilde-hypercube}
 with edges linking words at tilde-distance $1$.
Then, 
we introduce and study some isometric subgraphs of the tilde-hypercube 
obtained by using special words called {\em tilde-isometric words}. The subgraphs keep 
only the vertices that avoid a given tilde-isometric word as a factor. In the case of word $11$, 
the subgraph is called
 {\em tilde-Fibonacci cube}, as a generalization of the classical Fibonacci cube.  
 The tilde-hypercube and the tilde-Fibonacci cube can be recursively defined; the same holds for the number of their edges. This allows an asymptotic estimation of 
 the number of  edges in the tilde-Fibonacci cube,
in comparison to the total number in the tilde-hypercube.


%
\end{abstract}

\section{Introduction}\label{s-intro} 
 The $n$-dimensional {\em hypercube}, $Q_n$, encloses all the binary strings of length $n$ and hence it is a model that deserves a starring role in graph theory.   It is defined as a graph whose vertices are in correspondence with the $2^n$ words of length $n$  and there is an edge between two vertices if the corresponding words differ in one position, that is if their {\em Hamming distance} is 1. Hence, the distance between  two vertices in the graph  is equal to the Hamming distance of the corresponding words. 
During the years, the notion of hypercube has been extensively investigated (see \cite{HararyHW88} for a survey). Hypercubes are  used for designing interconnection networks and they found  applications also in theoretical chemistry (see  \cite{K13} for a survey). However, hypercubes have a critical limitation due to the fact that they have an exponential number of vertices. For this, various mo\-di\-fi\-cations have been proposed by considering subgraphs that are {\em isometric}, that is the distance of any pair of vertices in such subgraphs is the same as the distance in the complete hypercube.  With this aim, in 1993, Hsu
 introduced the {\em Fibonacci cubes} \cite{Hsu93}.  They are isometric subgraphs of $Q_n$ obtained by selecting only the vertices whose  corresponding words   do not contain $11$ as factor. They have many remarkable properties also related to Fibonacci numbers.

Generalized Fibonacci cubes $Q_n(f)$  were introduced in 2012 as the subgraphs of $Q_n$ 
keeping only
vertices associated to binary words that do not contain $f$ as a factor, i.e. $f$-free binary words \cite{IlickKR2012}.
Note that, in order to get an isometric subgraph of $Q_n$, the avoided word should satisfy some special conditions; if this is the case, then the word is said isometric. Indeed, 
a binary word $f$ is {\em isometric} (or Ham-isometric) when, for any $n\geq 1$,
$Q_n(f)$ can be isometrically embedded into $Q_n$, and {\em non-isometric}, otherwise  \cite{KlavzarS12}. 
The structure of binary Ham-isometric words has been characterized in \cite{IlicK12,KlavzarS12,Wei17,WeiYZ19,Wei16} and 
the research on the topic is still very active \cite{Azarija16,Wei21,WeiYW20}.

Recently, binary Ham-isometric words have been considered in the two-dimensional setting, and Ham-non-isometric pictures (also called \emph{bad pictures}) have been investigated \cite{AGMS-DCFS20}.
Moreover, the notion of isometric word has been extended to the case of alphabets of size $k$, with $k >2$, by considering  $k$-ary $n$-cubes, $Q^k_n$, and $k$-ary $n$-cubes avoiding a word $f$, $Q^k_n(f)$. 
In this setting, the distance between two vertices is no longer their Hamming distance, but their Lee distance. Taking into account this distance, Lee-isometric $k$-ary words have been introduced, studied and characterized \cite{AFMWords21,MMM22,AnselmoFM22}. 
Using the characterizations of Ham- and Lee-isometric words, in \cite{MMM22,BealC22}, some linear-time algorithms
are provided in order to check whether a word is  isometric and to give some interesting information on non-isometric words.
Worthily, 
Ham- and Lee-isometric words can be defined and studied by ignoring hypercubes and adopting a point of view closer to combinatorics on words. Actually, a 
word $f$ is Ham- (Lee-, resp.) isometric, if for any pair of $f$-free words $u$ and $v$ of the same length, $u$ can be transformed in $v$ by a sequence of $f$-free words, starting with $u$ and ending with $v$, such that the sequence has length equal to the Hamming (Lee, resp.) distance between $u$ and $v$ and every two consecutive words in the sequence have Hamming (Lee, resp.) distance equal to $1$.

In some applications coming from computational biology,  it seems natural to consider  the {\em swap} operation of exchanging two adjacent different symbols in a word. Then, an edit  distance based on swap and mismatch errors seems worth considering \cite{AmirEP06,FaroP18}. In \cite{donz1} this distance is referred to as   {\em tilde-distance}, since the $\sim$ symbol somehow evokes the swap operation. Tilde-isometric words have been defined using the tilde-distance, in place of Hamming or Lee distance, and
 studied from a combinatorial point of view.

In this paper, the tilde-distance serves as the base to define  the  {\em tilde-hypercube}, $\tilde Q_n$; it has again all the $n$-binary strings as vertices, but the edges correspond to tilde-distance equal to 1.  This implies that $\tilde{Q}_n$ has more edges than $Q_n$; in particular, since a swap corresponds to two mismatches, some vertices having distance $2$ in $Q_n$,  become adjacent in $\Tilde{Q}_n$. 
 We give a recursive construction of tilde-hypercubes and enumerate the number of their edges. Then, we consider  subgraphs of the tilde-hypercubes $\tilde Q_n(f)$ by selecting the vertices corresponding $f$-free words, for a given word $f$. 
 It is easy to show that $f$ is {\em tilde-isometric} if and only if $\tilde Q_n(f)$ is an isometric subgraph of $\tilde Q_n$. 
We present an infinite family of tilde-isometric words that are not Hamming isometric.  The last part of the paper is devoted to select special words $f$. 
 For what concern the word $f=11$, that is both Hamming- and tilde-isometric, the subgraph $\tilde Q_n(11)$ is referred to as  the {\em tilde-Fibonacci cube}. We present a recursive construction for it and we compare it with the classic Fibonacci cube.  We show that
the number of edges in the tilde-Fibonacci cube is about 1/7 less than the number of
edges in the whole tilde-hypercube. We also examine  $Q_n(1010)$, where $1010$ is a tilde-non-isometric but Ham-isometric word and $\tilde Q_n(11100)$, where $11100$ is Ham-non-isometric and tilde-isometric word. 
The paper ends with a small table comparing vertices and edges cardinality and ratio of $Q_n(1010)$ and $\tilde Q_n(11100)$ of order n = 4, . . . , 16.

\section{Preliminaries}

 In this paper we only focus on the binary alphabet $\Sigma=\{0, 1\}$.
A word (or string) $w$ of length $|w|=n$, is $w=a_1a_2\cdots a_n$, 
where $a_1, a_2, \ldots, a_n$ are symbols in $\Sigma$. The set of all words over $\Sigma$ is denoted $\Sigma^*$.
Finally, $\epsilon$ denotes the {\em empty word} and $\Sigma^+=\Sigma^* - \{\epsilon\}.$ 
For any word $w=a_1a_2\cdots a_n$, the {\em reverse} of $w$ is the word $w^{rev}=a_na_{n-1}\cdots a_1$. If $x\in\Sigma$, $\overline{x}$ denotes the opposite of $x$, i.e $\overline{x}=1$ if $x=0$ and viceversa. 
Then we define {\em complement} of $w$ the word $\overline{w}=\overline{a}_1\overline{a}_2\cdots \overline{a}_n$.

Let 
  $w[i]$ denote the symbol of $w$ in position $i$, i.e. $w[i]=a_i$.
  Then, $w[i .. j] = a_i \ldots a_j$,
   for $1\leq i\leq j\leq n$, denotes a \emph{factor}  of $w$. 
  The \emph{prefix}  (resp. \emph{suffix}) of $w$ of length $l$, with $1 \leq l \leq n-1$ is ${\rm pre}_l(w) = w[1 .. l]$  (resp. ${\rm suf}_l (w) = w[n-l+1 .. n]$).
  When ${\rm pre}_l(w) = {\rm suf}_l (w)=u$ then $u$ is here referred to as an \emph{overlap} of $w$ of length $l$; 
  in other frameworks, it is also called border, or bifix. A word $w$ is said {\em $f$-free} if $w$ does not contain $f$ as a factor.

An {\em edit operation} is a function $O: \Sigma^* \to \Sigma^*$ that transforms a word into another one. 

Let $OP$ be a {\em set of edit operations}. The {\em edit distance} of  words $u, v \in\Sigma^*$ 
is the minimum number of edit operations in $OP$ needed to transform $u$ into $v$. 
 
In this paper, we consider the edit distance that uses only {\em swap} and {\em replacement} operations to fix {\em swap} and {\em mismatch} errors. Note that these  operations preserve  the length of the word. 
\begin{definition}
    Let $w=a_1a_2\ldots a_n$ be a word over $\Sigma$.
    \\The {\em replacement operation} (or {\em replacement}, for short) on $w$ at position $i$ 
    
is defined by
   
   $$R_{i}(a_1a_2\ldots a_{i-1}\bm{a_i}a_{i+1}\ldots a_n)=a_1a_2\ldots a_{i-1} {\bm{\overline{{a}}_i}} a_{i+1}\ldots a_n.$$
   
    \noindent The {\em swap operation} (or {\em swap}, for short) on $w$ at position $i$ 
 
   with $a_i \neq a_{i+1}$, is defined by
    $$S_i(a_1a_2\ldots a_{i-1}\bm{a_ia_{i+1}}a_{i+2}\ldots a_n)=a_1a_2\ldots a_{i-1} \bm{ a_{i+1}a_i}a_{i+2} \ldots a_n.$$
    
     \noindent 
    
\end{definition}

Note  that one swap corresponds to the replacement of two consecutive symbols.

The {\em Hamming distance} ${\rm dist}_H(u,v)$ of $u,v\in \Sigma^*$ is defined as the minimum number of replacements needed to get $v$ from $u$.
A word $f$ is {\em Ham-isometric} if for any pair of $f$-free words $u$ and $v$, there exists a 
sequence of replacements  of length ${\rm dist}_H(u,v)$  that transforms $u$ into $v$ where all the intermediate words are also $f$-free. 

A word $w$ has a {\em $2$-error overlap} if  there exists $l$ such that 
${\rm pre}_l (w)$ and ${\rm suf}_l (w)$ have Hamming distance $2$ (cf. \cite{Wei17}). Then, it is proved  
the following characterization of Ham-isometric words. 

\begin{proposition}[\cite{Wei17}]\label{p-Hamiso}
   A word $f$ is Ham-isometric if and only if $f$ has a 2-error overlap.
\end{proposition}

Let $G$ be a graph,  $V(G)$ be the set of its nodes and $E(G)$ be the set of its  edges. The distance of 
$u,v\in V(G)$, $\dist_G(u,v)$, is the length of the shortest path connecting $u$ and $v$ in $G$. 
The {\em diameter} of $G$, denoted by $d(G)$, is the maximum distance of two vertices in $G$.
A subgraph $S$ of a (connected) graph $G$ is an {\em isometric subgraph} if for
any $u,v \in V (S)$, $\dist_S(u,v)=\dist_G(u,v)$.

Let us recall the notion of hypercube and Fibonacci cube, related to the Hamming distance. 
The {\em $n$-hypercube}, or binary $n$-cube, $Q_n$,
is a graph with $2^n$ vertices, each associated to a binary word of length $n$. The vertices are often identified with  the associated word.
Two vertices $u$ and $v$ 
in $Q_n$ are adjacent when their associated words differ in exactly $1$ position, i.e. when 
$\dist_H(u,v)=1$.
Therefore, $\dist_{Q_n}(u,v)=\dist_H(u,v)$. 

Denote by $f_n$ the $n$-th Fibonacci number, defined by $f_1=1, f_2=1$ and $f_i=f_{i-1}+f_{i-2}$, for $i\geq 3$.
The {\em Fibonacci cube} $F_n$ of order $n$ is the subgraph of $Q_n$ whose vertices are binary words of length $n$ avoiding the factor $11$. It is well known that $F_n$ is an isometric subgraph of $Q_n$ (cf. \cite{K13}). Isometric subgraphs of hypercubes are also called {\em partial cubes}.

One of the main properties of $Q_n$ and $F_n$ is their {\em recursive structure} that have been extensively studied (cf. \cite{Hsu93}, \cite{MunariniS02} and \cite{K13}).

The following results are well-known, but are hereby stated for future reference. 

\begin{proposition}\label{p-vert-edge-hyper}
 Let $Q_n$ be the hypercube of order $n$ and $F_n$ be the Fibonacci cube.
Then
\begin{itemize}
    \item $\vert V(Q_n)\vert =2^n$ and 
 $\vert E(Q_n)\vert =n2^{n-1}$
\item $\vert V(F_n)\vert =f_{n+2}$
 \item $\vert E(F_1)\vert =1 $, $\vert E(F_2)\vert =2$ and 
$\vert E(F_n)\vert =\vert E(F_{n-1})\vert + \vert E(F_{n-2})\vert +f_{n}, \forall n >2$ 
$$\vert E(F_n)\vert =\frac{2(n+1)f_n+nf_{n+1}}{5}$$ 
\end{itemize}
 \end{proposition}

The sequence $|E(F_n)|$ is Sequence A001629 in \cite{Sloane}. Hence, the number of edges of a Fibonacci cube with $N$ vertices is $O(N\log N)$, asymptotically equal to the number of edges of a hypercube with the same number of vertices.

\section{Tilde-isometric words}
In this section, we consider the edit distance based on swap and replacement operations used to fix swap and mismatch errors between two words. It is called tilde-distance and denoted by $dist_\sim$. We  recall the definition of tilde-isometric words given in \cite{donz1} and then present a family of tilde-isometric words.

\begin{definition}\label{def:d-tilde-distance}
	Let $u, v \in \Sigma^*$ be words of equal length.
	The {\em tilde-distance} ${\rm dist}_\sim(u,v)$ between $u$ and $v$ is the minimum number of replacements and swaps needed to transform $u$ into $v$.
	\end{definition}
\begin{definition}\label{def:tilde-transformation}
  	Let $u, v \in \Sigma^
   *$ be words of equal length. 
   \\
  	A \emph{tilde-transformation} $\tau$
  	of length $h$ from $u$ to $v$ is a sequence of words $(w_0, w_1, \ldots, w_h)$ such that
  	$w_0=u$, $w_h=v$, and for any $k=0, 1, \ldots ,h-1$,
  	${\rm dist}_\sim(w_k,w_{k+1})=1$.  Further, $\tau$ is 
   {\em $f$-free} if for any $i = 0,1, \ldots , h$,  word $w_i$ is $f$-free. 
It
is {\em minimal} if its length is equal to ${\rm dist}_{\sim}(u,v)$ and characters in each position are modified at most once.   
  \end{definition}
    A tilde-transformation $(w_0, w_1, \ldots, w_h)$ from $u$ to $v$  {is associated to a} sequence of $h$ operations $(O_{i_1}, O_{i_2},\ldots O_{i_h})$  such that, for any $k=1, \ldots ,h$,  $O_{i_k} \in \{R_{i_k},S_{i_k}\}$ and $w_{k}=O_{i_k}(w_{k-1})$; 
 it can be represented as follows: 
 $$u=w_0\xrightarrow{O_{i_1}}w_1 \xrightarrow{O_{i_2}} \cdots \xrightarrow{O_{i_h}}w_h=v.$$
 With a little abuse of notation,  in the sequel we will refer to a tilde-transformation both  as a sequence of words and as a sequence of operations. 
Let us give some examples.
\begin{example}\label{ex:tilde-trans} 
Let  $u=1011, v=0110$. Below, two different tilde-transformations from $u$ to $v$ are shown.  Note that the length of $\tau_1$ corresponds to ${\rm dist}_\sim(u,v)=2$.

$$\tau_1: 1011\xrightarrow{S_1}0111 \xrightarrow{R_4}0110\hspace{0.8cm}\tau_2:1011 \xrightarrow{R_1} 0011 \xrightarrow{R_2}0111 \xrightarrow{R_4}0110$$
Furthermore, consider the  following tilde-transformations of $u'=100$ into $v'=001$:
$$\tau'_1:100\xrightarrow{S_1}010 \xrightarrow{S_2}001 \hspace{1cm}\tau'_2:100 \xrightarrow{R_1} 000  \xrightarrow{R_3}001$$
 Note that both $\tau'_1$ and $\tau'_2$ have the same length equal to ${\rm dist}_\sim(u',v')=2$ and that, in $\tau'_1$ the symbol in position 2 is changed twice.
\end{example}

In \cite{donz1} it is proved that a minimal tilde-transformation always exists in the binary case. 
Let us now define isometric words based on the tilde 
distance.

\begin{definition}\label{d-tilde-iso}
	Let $f\in\Sigma^*$ be a word of length $n$ with $n\geq 1$. The word $f$ is \emph{tilde-isometric} if for any pair of $f$-free words $u$ and $v$ of equal length $m \geq n$, there exists a minimal tilde-transformation from $u$ to $v$ that is $f$-free.
    It is \emph{tilde-non-isometric} if it is not tilde-isometric. 
\end{definition}

In order to prove that a word is tilde-non-isometric it is sufficient to exhibit a pair $(u,v)$ of words contradicting  Definition \ref{d-tilde-iso}. More challenging is to prove that a word is tilde-isometric.

\begin{example}\label{e-1010}
The word $f=1010$ is tilde-non-isometric. In fact, let  $u=11000$ and $v=10110$; $u$ and $v$ are $f$-free; moreover the only possible minimal tilde-transformations from $u$ to $v$ are
$11000 \xrightarrow{S_2} 10100 \xrightarrow{R_4} 10110$ and $11000 \xrightarrow{R_4} 11010 \xrightarrow{S_2} 10110$, and in both cases $1010$ appears as factor after the first step.
On the other side, observe that $f$ is Ham-isometric by Proposition \ref{p-Hamiso}.
\end{example}

\begin{remark}\label{r-ham-replacements}
When a tilde-transformation contains a swap and a replacement that are adjacent, there could exist   minimal tilde-transformations that involve different sets of operations. 
For instance, the pair $(u,v)$, with $u=010$ and $v=101$, has two minimal tilde-transformations:
$010 \xrightarrow{S_1} 100 \xrightarrow{R_{3}} 101$ and $010 \xrightarrow{S_2} 001 \xrightarrow{R_{1}} 101$.

 This fact never occurs when only replacements are allowed and thus it constitutes a new difficulty, with respect to the Hamming distance case, to prove the isometricity. 
\end{remark}

Let us highlight the following straightforward property of tilde-isometric binary words that is very helpful to simplify proofs.

\begin{remark}\label{r-rev-comp}
 A word $f$ is tilde-isometric iff $\overline{f}$ is tilde-isometric iff $f^{rev}$
   is tilde-isometric.
\end{remark}
In view of Remark \ref{r-rev-comp} , we will focus on words starting with 1.
The following proposition explicitly explores the tilde-isometricity for all words of length $2$, $3$ and $4$. 
\begin{proposition}\label{p-words-iso}
The following statements hold.
\begin{enumerate}
\item All words of length $2$ are tilde-isometric
\item All words of length $3$, except for $101$ and $010$, are tilde-isometric
\item The words $1111$, $1110$, $1000$, $0000$, $0001$, $0111$ are tilde-isometric. All the other words of length $4$ are tilde-non-isometric.
\end{enumerate}
\end{proposition}
\begin{proof}
Consider the different cases:
\begin{enumerate}
  \item 
   If $|f|=2$, two cases arise, up to reverse and complement: $f=10$ and $f=11$. They are both tilde-isometric words because of next Proposition \ref{p-0h1k}.

     \item If $|f|=3$, three cases arise, up to complement and reverse. 
    
     \begin{itemize}
         \item $f=111$ and $f=100$, are tilde-isometric because of Proposition \ref{p-0h1k}.
         \item $f=101$, is tilde-non-isometric. In fact, $u=1111$ and $v=1001$ contradict isometricity of $f$, since they are $f$-free, and all the minimal transformations of $u$ into $v$ need to change $u$ into a word that has $101$ as a factor.
     \end{itemize}

     \item If $|f|=4$, the following cases  arise, up to complement and reverse.
    
    \begin{itemize}
     \item $f=1111$ and $f=1110$ are tilde-isometric because of Proposition \ref{p-0h1k}.

     \item $f=1100$. Then $u=110100$ and $v=101010$ contradict isometricity of $f$.

     \item $f=1001$. Then $u=11011$ and $v=10001$ contradict isometricity of $f$.

     \item $f=1010$ is tilde-non-isometric (see Example \ref{e-1010}).

     \item $f=1011$. Then $u=11111$ and $v=10011$ contradict isometricity of $f$.    
\end{itemize} 
    \end{enumerate}\end{proof}

Let us show an infinite family of words that are tilde-isometric, but not Ham-isometric, by Proposition \ref{p-Hamiso}.

\begin{proposition}\label{p-0h1k}
 Let 
 $f_{h,k}=1^h0^k$, 
 with $h, k \geq 0$.
 Then,  $f_{h,k}$ is tilde-isometric for any $h, k \geq 0$, except for $h=k=2$, i.e.,  $f_{2,2}=1100$ is tilde-non-isometric.
 \end{proposition}
\begin{proof}
Suppose that $f=f_{h,k}=1^h0^k$, $f \neq1100$,  is tilde-non-isometric and let $(u, v)$, with $u, v\in \Sigma^m$, be a pair of words contradicting  Definition \ref{d-tilde-iso}, with minimal $d={\rm dist}_\sim(u, v)$  among all such pairs of words with length $m$.

Let $\{O_{i_1}, O_{i_2}, \dots , O_{i_d}\}$ be the set of operations of a minimal tilde-transformation from $u$ to $v$,	
$1\leq i_1< i_2< \dots  < i_d\leq m$, where for any $j=1, 2, \dots, d$, $O_{i_j}\in \{ R_{i_j}, S_{i_j}  \}$. 
Then, for  $j=1, 2, \dots, d$, $O_{i_j}(u)$ has an occurrence of $f$ in an interval, say $I_j$, which contains at least one position modified by $O_{i_j}$. Note that,
 this occurrence of $f$ must disappear in a tilde-transformation from $u$ to $v$, because $v$ is $f$-free. Hence, 
 $I_j$
 contains a position modified by another operation in $\{O_{i_1}, O_{i_2}, \dots , O_{i_d}\}$.
By the pigeonhole principle, there exist $s,t \in \{i_1, i_2, \dots i_d\}$, such that $O_s(u)$ has an occurrence of $f$ in 
$I_s=[k_s .. k_s+n-1]$ that contains at least one position modified by $O_{t}$ and $O_t(u)$ has an occurrence of $f$ in 
$I_t=[k_t .. k_t+n-1]$ that contains at least one position modified by $O_{s}$. Without loss of generality,  suppose that $k_s <k_t$. Now, let $I=[p .. q]$ be the intersection of $I_s$ and $I_t$; this interval intercepts  a suffix of $f$ in $O_s(u)$ and a prefix of $f$ in $O_t(u)$ of  same length $l$, with $l=q-p+1$. In other words,  $(O_{s}(u))[p .. q]=suf_l(f)$ and  $(O_{t}(u))[p .. q]=pref_l(f)$. Note that this implies $f \neq 1^k$.

 The  interval $I$ can contain either four, or three, or two among the positions modified by $O_{s}$ and $O_{t}$, of which at least one is modified by $O_{s}$ and at least one by $O_{t}$.
One can show that a contradiction follows in all cases. We  give details only in some cases that involve swap operations; the other ones can be treated in an analogous way.

Consider the case that $I$ contains four positions modified by $O_{s}$ and $O_{t}$. Therefore, $O_{s}$ and $O_{t}$ are swaps, i.e. $O_{s}=S_s$ and $O_{t}=S_t$, with $s, t \in [p..q]$.
Since $(S_{s}(u))[p .. q]=suf_l(f)$, one has 

$u[p .. q] \in 1^*010^*$. But, then, there exists no other swap operation in $u[p .. q]$ that can give a prefix of $f$, as it should be for $S_t$.

Consider now the case that $I$ contains two positions modified by $O_{s}$ and $O_{t}$. 
Three cases are possible following that $O_{s}$ and $O_{t}$ are both replacement operations, or both swap operations

(with $s=p-1$ and $t =q$), or one is a swap and the other a replacement 
(with $s=p-1$ and $t \in [p+1 .. q]$ or $s \in [p .. q-1]$ and $t=q$).

Let us consider the case they are

both swap operations.
If $O_{s}=S_{p-1}$ and $O_{t}=S_q$ then the two positions modified in $u[p .. q]$ by $O_s$ and $O_t$ must be positions $p$ and $q$.

Suppose $q=p+1$.

If $u[p .. q]=10$ then $u[p-1 .. q+1]=0101$. The application of 
$S_{p-1}$ on $u$ implies that
 $f$ ends with $100$, whereas the application of $S_q$ implies that $f$ begins with $110$. Hence, $f=1100$, against the hypothesis. 
 The application of $S_{p-1}$ on $u$
 in the cases that 
 $u[p .. q]=00, 01, 11$, respectively, 
 would result in a suffix $010$, $011$, $101$ of $f$, and this is a contradiction.
 Suppose now $q > p+1$. 
 
 If $u[p]=1$ then $u[p .. q] \in 10^*$, since $(S_{p-1}(u))[p .. q]=suf_l(f)$, 
 but there exists no other swap operation in $u[p .. q]$ that can give a prefix of $f$, as it should be for $S_q$. An analogous reasoning shows that  $u[p]=0$ cannot hold either. Therefore, also in this case, a contradiction follows.
\end{proof}
The notion of tilde-isometricity is not comparable with the one of Ham-isometricity. Furthermore, the following result holds.

\begin{proposition}\label{p-short-til-iso}
    The word $11100$ is the shortest tilde-isometric word that is not Ham-isometric.
     The word $1010$ is the shortest Ham-isometric word that is not tilde-isometric.
\end{proposition}
\begin{proof}
 The word $11100$ is tilde-isometric (Proposition \ref{p-0h1k}) but Ham-non-isometric (Pro\-po\-si\-tion \ref{p-Hamiso}). On the other hand $1010$ is tilde-non-isometric (Proposition \ref{p-words-iso}), and it is Ham-isometric (Proposition \ref{p-Hamiso}). The minimality of the length of these words comes from Proposition \ref{p-words-iso}.
\end{proof}

\section{The tilde-hypercube} \label{s-tilde-hyper}
Classical hypercubes connect vertices following their Hamming distance, whereas the distance of vertices in a $k$-ary $n$-cube represents their Lee-distance.
This suggests to investigate hypercubes based on other distances. In this paper
we introduce the {\em tilde-hypercube}, whose vertices are the binary words and edges connect vertices with tilde-distance equal to $1$. Then, its recursive structure is explored.

\begin{definition}
The \emph{$n$-tilde-hypercube} ${\tilde{Q}}_n$,
is a graph with $2^n$ vertices, each associated to a binary word of length $n$.
Two vertices 
in ${\tilde{Q}}_n$, are adjacent whenever their tilde-distance is $1$.
\end{definition}

\begin{figure}[]
\begin{center}
    \begin{subfigure}{0.4\textwidth}
\includegraphics[width=150pt]{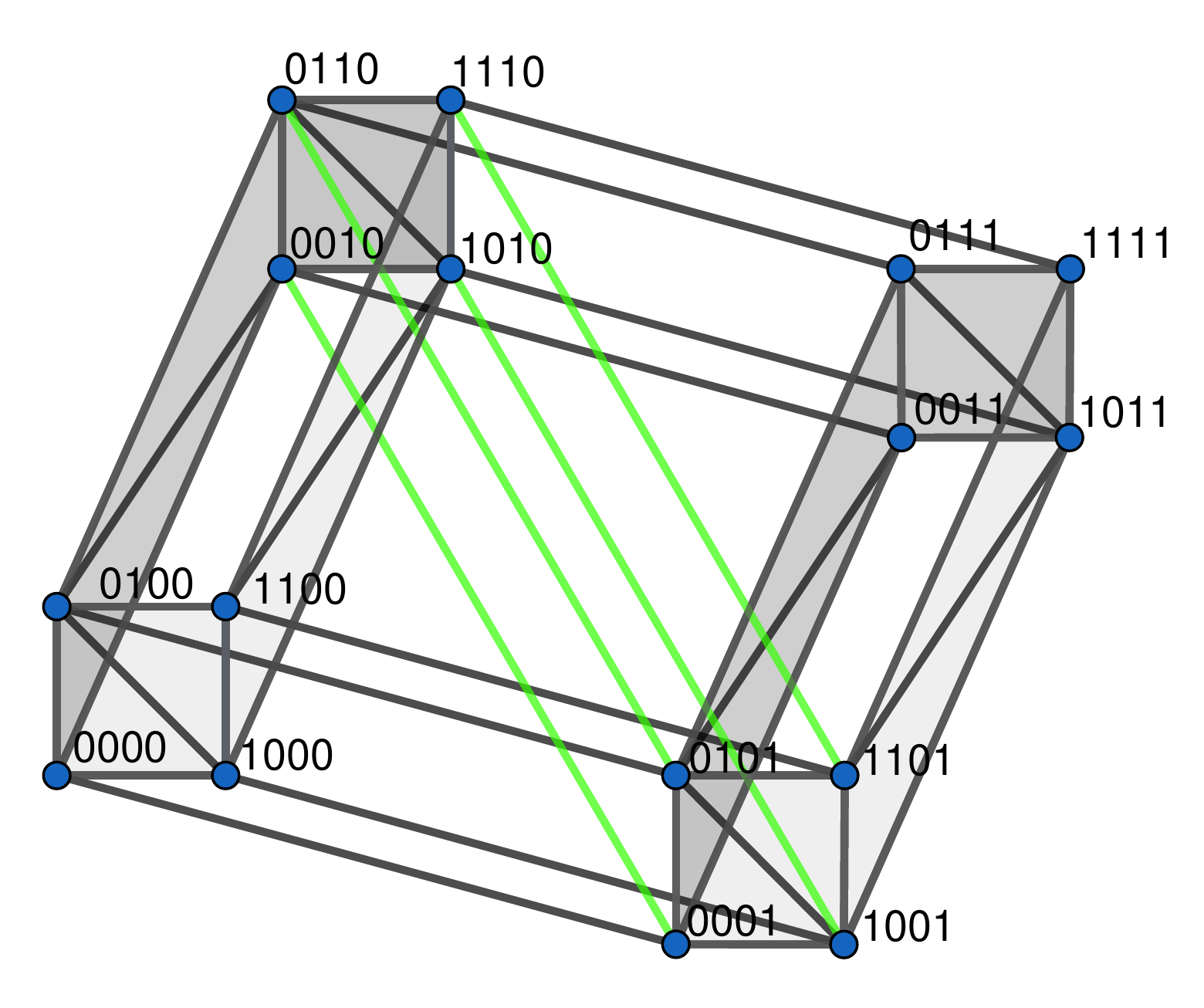}
\subcaption[(a)]{}
\end{subfigure}
\hspace{.4cm}
\begin{subfigure}{0.4\textwidth}
     \includegraphics[width=130pt]{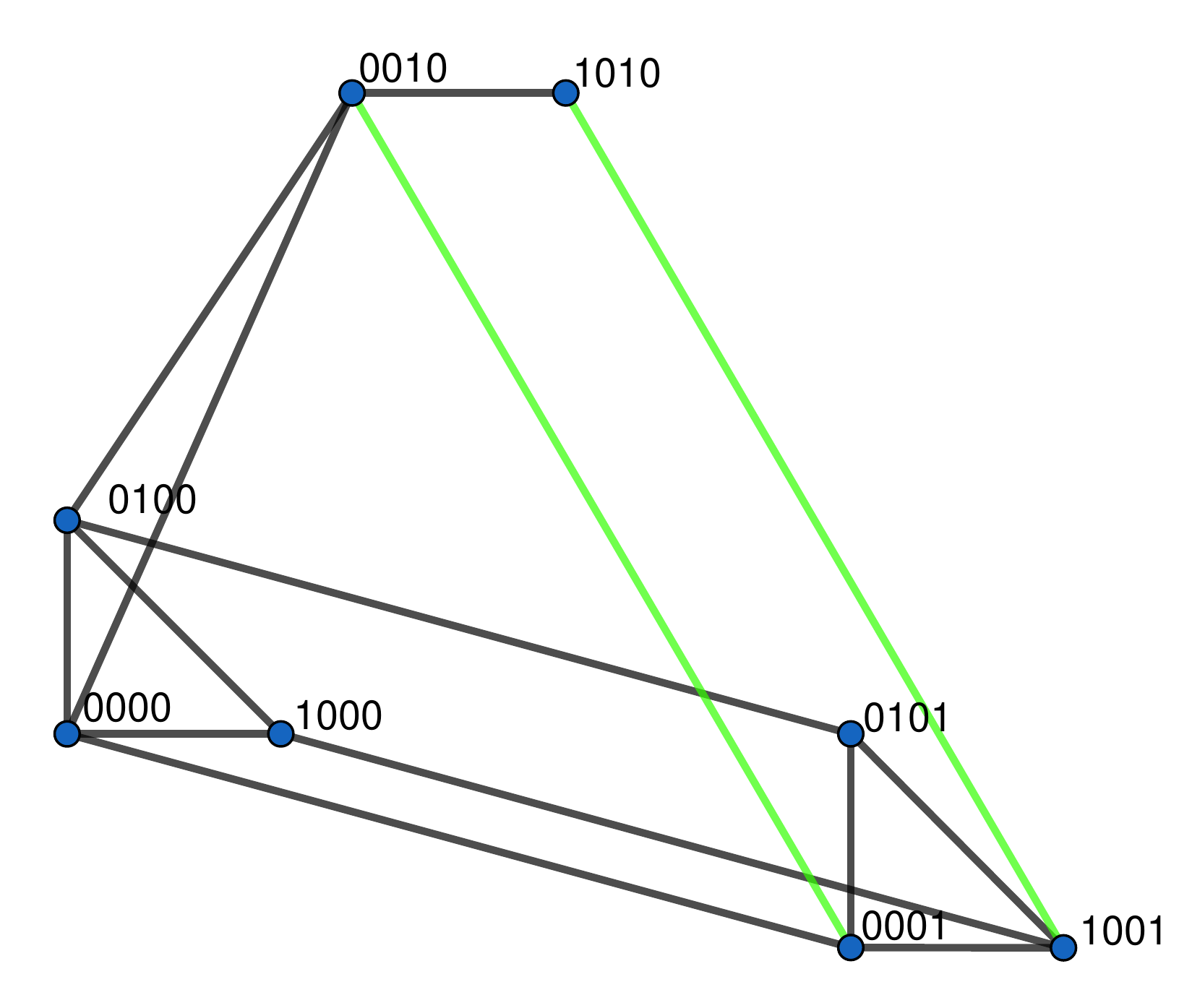} 
     \subcaption[(b)]{}
\end{subfigure}
    \caption{The tilde-hypercube of order $4$ (a)  the tilde-Fibonacci cube of order $4$ (b)}\label{f-overlap}\label{f-tilde-hyper}
\end{center}
\end{figure}

Figure \ref{f-tilde-hyper}(a) shows the tilde-hypercube of order $4$.

\begin{remark}
$Q_n$ is a proper subgraph of ${\tilde{Q}}_n$. In fact for $u,v\in \Sigma^*$, $dist_H(u,v)=1$ implies $dist_\sim(u,v)=1$. 
Further, for any $n\geq 2$, there exist words $u_n, v_n$ of length $n$ such that $dist_\sim(u_n,v_n)=1$ and $dist_H(u_n,v_n)\neq 1$, for example $u_n=0^{n-2}01$ and $v_n=0^{n-2}10$, so that $(u_n,v_n)$ is an edge in ${\tilde{Q}}_n$ but not in $Q_n$.
\end{remark}

The following lemma is the main tool to exhibit a recursive definition of the tilde-hypercube, in analogy with the classical hypercube.

\begin{lemma}\label{l-extension}
For any $u, v \in \Sigma^{n-1}$, $dist_\sim(u0,v0)= dist_\sim(u,v)=dist_\sim(u1,v1)$ and $dist_\sim(u0,u1)= 1$. Moreover for any $u^\prime \in\Sigma^{n-2}$, $dist_\sim(u^\prime01, u^\prime10)=1$.
\end{lemma}

\begin{proposition}\label{p-tilde-hyper-rec}
$\tilde{Q}_n$ can be recursively defined.
\end{proposition}
\begin{proof}
If $n=1$, $\tilde{Q}_1$ has just two vertices $0$ and $1$ connected by an edge.

Suppose  the tilde-hypercubes of dimension smaller than $n$ have been defined. Consider two copies of $\tilde{Q}_{n-1}$. In the first copy all the vertices $u$ are replaced by $u0$ and in the second by $u1$. By Lemma \ref{l-extension}, if $u$ and $v$ are connected in $\Tilde{Q}_{n-1}$, then $u0$ and $v0$ ($u1$ and $v1$, respectively) are connected in $\Tilde{Q}_n$. Moreover for any $u\in \Sigma^{n-1}$, $u0$ in the first copy and $u1$ in the second copy are linked. Finally, for each vertex of $\tilde Q_{n-1}$ that ends with 1, say $u=u^\prime 1$, there is an edge between $u^\prime 10$ in the first copy of $\Tilde{Q}_n$ and $u^\prime 01$ in the second copy of $\Tilde{Q}_n$ (see green edges in Fig. \ref{f-tilde-hyper}(a)). For any other pair of words $u,v\in \{0,1\}^n$ we have $dist_\sim (u,v)>1$.
\end{proof}


\begin{corollary}\label{c-edges-tilde-hyper}
   Let $\Tilde{Q}_n$ be the tilde-hypercube of order $n$. Then 
   $$\vert E(\Tilde{Q}_n)\vert=2\vert E(\Tilde{Q}_{n-1}) \vert+2^{n-1}+2^{n-2}, \,\,\, \mbox {with} \,\, \vert E(\Tilde{Q}_1) \vert=1$$
\end{corollary} 
\begin{proof}
By the recursive construction in Proposition \ref{p-tilde-hyper-rec}, $\Tilde{Q}_n$ has twice the number of edges of $\Tilde{Q}_{n-1}$ (since it has two copies of it),  plus $2^{n-1}$ edges, one for each vertex of $\Tilde{Q}_{n-1}$, plus $2^{n-2}$ edges, one for each vertex of 
$\Tilde{Q}_{n-1}$ that ends with a $1$. 
\end{proof}

By solving the recurrence we find the exact solution $\vert E(\tilde Q_n )\vert=(3n-1)\cdot2^{n-2}$ (Sequence A053220 in \cite{Sloane}). Let $\tilde{EQ}(N)$ be the number of edges of the  tilde-hypercube with $N$ vertices. Then,
\begin{equation}\label{eq-tilde-hyper}
  \tilde{EQ}(N)=N(3\log N -1)/4  
\end{equation} 


\section{The tilde-hypercube avoiding a word}
The so-called generalized Fibonacci cube has  been defined in \cite{IlickKR2012} 
as the subgraph of the hypercube where the vertices having a given word as factor are removed. 
In analogy, we introduce the definitions of the tilde-hypercube and the

tilde-Fibonacci cube.

\begin{definition}
	The {\em  $n$-tilde-hypercube avoiding a word $f$}, denoted $\tilde{Q}_n(f)$, is the subgraph of $\tilde{Q}_n$ obtained by removing those vertices which contain $f$ as a factor.
\end{definition}

Next proposition states the relationship between tilde-isometric words and subgraphs of the tilde-hypercube avoiding a word. The proof can be easily derived from the definitions.
\begin{proposition}\label{d-Cube-iso-word}
	A word $f\in\Sigma^*$ is tilde-isometric if and only if for all  {$n \geq |f|$}, 
 $\tilde{Q}_n(f)$ is an 
 isometric subgraph of $\tilde{Q}_n$.
\end{proposition}

\begin{example}
All binary words of length $2$ are both tilde-isometric (Proposition \ref{p-words-iso}) and Ham-isometric and for each $n\geq 1$, $\Tilde{Q}_n(10)$ and $Q_n(10)$ coincide. In fact, $V(\Tilde{Q}_n(10))=\{0^h1^k|\,\, h,k \geq 0, \hspace{0,2cm} h+k=n \}$ and $E(\Tilde{Q}_n(10))=\{(0^i1^j,0^{i-1}1^{j+1})| \,\, 1 \leq i,j \leq n\}$.
The case of word $11$ deserves to be treated in a separated section.
The other words of length $2$ are tilde-isometric by complement (see Remark \ref{r-rev-comp}).

\end{example}

\subsection{The tilde-Fibonacci cube}\label{ss-tilde-Fibo}
The tilde-hypercube avoiding word $11$ is called the tilde-Fibonacci cube, in analogy to the Fibonacci cube introduced by Hsu \cite{Hsu93}.
Here, we show a recursive construction of the tilde-Fibonacci cube; it allows to enumerate the number of its edges and then to compare it with the number of edges of the tilde-hypercube with the same number of vertices. 

\begin{definition}
The {\em $n$-tilde-Fibonacci cube}, denoted $\Tilde{F}_n$, is
$\Tilde{F}_n = \Tilde{Q}_n(11)$,
$n \geq 1$.
 \end{definition}

By Proposition \ref{p-vert-edge-hyper}, 
 $\vert V(\Tilde{F}_n) \vert=\vert V(F_n) \vert= f_{n+2}$. Among these vertices, $f_{n+1}$ end with a $0$ and $f_{n}$ end with a $1$. Figure \ref{f-tilde-hyper}(b) shows the tilde-Fibonacci cube of order $4$.

\begin{remark}\label{r-fibo-word}
Let $u\in V(F_{n-1})$,  $x \in \Sigma$. If  $u$ ends with $1$, then $ux \in V(\tilde{F}_n)$ iff $x=0$. If $u$ ends with $0$ then $ux \in V(\tilde{F}_n)$, for any $x \in \{0,1\}$.   
\end{remark}
 
\begin{proposition}\label{p-tilde-fibo-rec}
The tilde-Fibonacci cube $\tilde{F}_n$ can be recursively defined.
\end{proposition}
\begin{proof}
   If $n=1$, $\tilde{F}_1$ has two vertices $0$ and $1$ connected by an edge. 
   If $n=2$, $\tilde{F}_2$ has three vertices $00$, $01$ and $10$ and $E(\tilde{F}_2)=\{(00,10), (00,01), (01,10)\}$.

Suppose  $\tilde{F}_i$ are defined for all $i<n$. For any $n\geq 3$, $\tilde{F}_n$ can be constructed from a copy of $\tilde{F}_{n-1}$ (say $\tilde{\bf F}_{n-1}$) and a copy of $\tilde{F}_{n-2}$ (say $\tilde{\bf F}_{n-2}$), where each vertex $u$ in $\tilde{F}_{n-1}$ is replaced by $u0$ in $\tilde{\bf F}_{n-1}$, and each vertex $v$ in $\tilde{F}_{n-2}$ is replaced  by $v01$ in $\tilde{\bf F}_{n-2}$. Further, for any $v$ of length $n-2$, there is an edge between $v00$ in $\tilde{\bf F}_{n-1}$ and $v01$ in $\tilde{\bf F}_{n-2}$, an edge between $v10$ in $\tilde{\bf F}_{n-1}$ and $v01$ in $\tilde{\bf F}_{n-2}$ (see the green edges in Fig. \ref{f-tilde-hyper}). By Remark \ref{r-fibo-word} and Lemma \ref{l-extension} no further edges exist in $\tilde{F}_n$.
\end{proof}

\begin{corollary}\label{c-edges-tilde-fibo}
Let $\tilde F_n$ be the tilde-Fibonacci cube. Then $\vert E(\Tilde F_1)\vert =1, \vert E(\Tilde F_2)\vert =3$ and 
$$\vert E(\Tilde F_n)\vert =\vert E(\Tilde F_{n-1})\vert + \vert E(\Tilde F_{n-2})\vert + f_{n+1},
\forall n \geq 2.$$ 
\end{corollary}
\begin{proof}
From the proof of Proposition \ref{p-tilde-fibo-rec}, $\Tilde F_1$ has $1$ edge and $\Tilde F_2$ has $3$ edges.
Moreover, $\vert E(\Tilde F_n)\vert $ is the sum of $\vert E(\Tilde F_{n-1})\vert $ with $\vert E(\Tilde F_{n-2})\vert $, plus one edge for each vertex in $\tilde{\bf F}_{n-1}$, i.e. $f_{n+1}$, by Proposition. \ref{p-vert-edge-hyper}.
\end{proof}
By solving the recurrence in Corollary \ref{c-edges-tilde-fibo}, we find the following exact solution 
$$\vert E(\Tilde F_n)\vert=\frac{(n+1)f_{n+3} + (n-2)f_{n+1}}{5}$$
(Sequence A023610 in \cite{Sloane} for $\vert E(\Tilde F_{n+1}\vert$).

Since the number of vertices of $\tilde F_n$ is $f_{n+2}$, from the previous formula it follows that the tilde-Fibonacci cube has $O(N\log N)$ edges, where $N$ is the number of vertices, as for the tilde-hypercube (see Equation (\ref{eq-tilde-hyper})).

To compare the number of edges of the Fibonacci cube and the hypercube, in \cite{Hsu93} the authors prove that the ratio between the number of edges $EF(N)$  and $EQ(N)$ in the Fibonacci cube and the  hypercube with $N$ vertices, respectively,  is asymptotically bounded by $0.79 < EF(N)/EQ(N) < 0.80$.  
In a\-na\-lo\-gy with this result, by using the same method as in \cite{Hsu93}, we have the following corollary.

\begin{corollary}
Let $\Tilde{EF}(N)$ and $\tilde{EQ}(N)$ be the number of edges of the tilde-Fibonacci cube and of the tilde-hypercube with $N$ vertices, respectively. Then, their ratio is asymptotically bounded by

$$0.85 < \frac{\tilde {EF}(N)}{\tilde {EQ}(N)} < 0.86$$ 

\end{corollary}
\begin{proof}
    By Equation (\ref{eq-tilde-hyper}) and Proposition \ref{p-vert-edge-hyper},  $\Tilde{EQ}(f_{n+2})=f_{n+2}(3\log f_{n+2}-1)/4$ and by Corollary \ref{c-edges-tilde-fibo},  
    $\tilde{EF}(f_{n+2})=((n+1)f_{n+3} + (n-2)f_{n+1})/5$. 
    By considering $\Tilde{EF}(f_{n+2})/\Tilde{EQ}(f_{n+2})$ asymptotically, the thesis follows.
\end{proof}

This proves that the number of edges of the tilde-Fibonacci cube is about $1/7$ less than the number of edges of the tilde-hypercube, with fixed number of vertices. The ratio is just slightly higher than in the Hamming case. This fact is not surprising  because the swap operation adds new edges, but, on the other hand, it shortens the average distances because a swap corresponds to two replacement operations. More formally, we have the following remark.

\vspace{-0.1cm}

\begin{remark}
In \cite{Hsu93} it is proven that the diameter $d(F_n)$ is $n$ and that the maximal distance involves the words  $(10)^{n/2}$ and $(01)^{n/2}$ for even $n$, and $(01)^{\lfloor n/2 \rfloor} 0$ and $(10)^{\lfloor n/2 \rfloor}1$ for odd $n$. If the tilde-distance is considered, then $d(\tilde F_n)=\lceil n/2 \rceil$. Indeed, the same words have maximal tilde-distance and the minimal tilde-transformation from one to the other  consists of $n/2$ swaps for even $n$ and $\lfloor n/2 \rfloor$ swaps and one replacement for odd $n$.

\end{remark}

\vspace{-0.5cm}

\begin{table}\small
    \centering
    \scalebox{0.68}{
    {\renewcommand{\arraystretch}{1.3}%
    \begin{tabular}{|c|c|c|c|c|c|c|c|c|c|c|c|c|c|}
    \hline 
    \emph{n} & 4 & 5 & 6 & 7 & 8 & 9 & 10 & 11 & 12 & 13 & 14 & 15 & 16 \\ 
    \hline
    $\vert V(Q_n)\vert$=$\vert V(\Tilde Q_n)\vert$ & 16 & 32 & 64 & 128 & 256 & 512 & 1024 & 2048 & 4096 & 8192 & 16384 & 32768 & 65536 \\ 
    \hline
    $\vert E(Q_n)\vert$ & 32 & 80 & 192 & 448 & 1024 & 2304 & 5120 & 11264 & 24576 & 53248 & 114688 & 245760 & 524288 \\ 
    \hline
    $\vert E(\Tilde Q_n)\vert$ & 44 & 112 & 272 & 640 & 1472 & 3328 & 7424 & 16384 & 35840 & 77824 & 167936 & 360448 & 770048 \\ 
    \hline
    $\vert V(\Tilde Q_n(11100))\vert$ & 16 & 31 & 60 & 116 & 224 & 432 & 833 & 1606 & 3096 & 5968 & 11504 & 22175 & 42744 \\
    \hline
    $\vert E(\Tilde Q_n(11100))\vert$ & 44 & 106 & 245 & 550 & 1208 & 2609 & 5569 & 11773 & 24691 & 51440 & 106566 & 219696 & 451005 \\
    \hline
    $R_\sim \ (11100)$ & 1 & 0,99 & 0,98 & 0,97 & 0,96 & 0,96 & 0,95 & 0,95 & 0,944 & 0,941 & 0,939 & 0,937 & 0,935 \\
    \hline
    $\vert V(Q_n(1010))\vert$ & 15 & 28 & 53 & 100 & 188 & 354 & 667 & 1256 & 2365 & 4454 & 8388 & 15796 & 29747 \\
    \hline
    $\vert E(Q_n(1010))\vert$ & 28 & 62 & 138 & 299 & 632 & 1323 & 2746 & 5645 & 11520 & 23377 & 47192 & 94830 & 189808 \\
    \hline
    $R_H \ (1010)$ & 0,96 & 0,92 & 0,91 & 0,90 & 0,89 & 0,88 & 0,88 & 0,87 & 0,869 & 0,866 & 0,863 & 0,861 & 0,859 \\
    \hline
   
    \end{tabular}}}
    \\
    \medskip
    \captionof{table}{Vertices and edges cardinality and ratio of some cubes of order $n= 4, \ldots ,16$} 
    \label{t-tilde-ver-edg}
\end{table}

\section{Conclusion and future work}
In this paper we have introduced the tilde-hypercube and the tilde-Fibonacci cube as a generalization of the  corresponding classical notions, with the tilde-distance in place of the Hamming one.

We have shown that, as in the classical case, the tilde-hypercube and the tilde-Fibonacci cube can be recursively defined. 

This made it possible to provide recursive and closed formulas for their number of edges with respect to the order. We used such results to quantify how many edges the tilde-Fibonacci cube has compared to the tilde-hypercube with the same number of vertices, and it turned out that his value is very close to the classical case.
However, the investigation  definitely deserves some deepening, since the hypercubes and the tilde-hypercubes are defined on different distances, which are supposed to be used for different applications.

Further, we have considered the hypercubes avoiding some special words, i.e., $1010$ and $11100$. The hypercube $Q_n(1010)$  is an isometric subgraph of $Q_n$, whereas $\tilde Q_n(1010)$ is not an isometric subgraph of $\tilde Q_n$.
On the contrary, the tilde-hypercube $\tilde Q_n(11100)$ is a 
isometric subgraph of $\tilde Q_n$, whereas $Q_n(11100)$ is not an isometric subgraph of $Q_n$ (cf. Propositions \ref {p-short-til-iso} and \ref {d-Cube-iso-word}).
Table \ref{t-tilde-ver-edg} resumes the first values of the number of vertices and edges of $Q_n$, $\tilde Q_n$, $\tilde Q_n(11100)$ and $Q_n(1010)$. Furthermore, for each $n$, the value of $R_H (1010)$ at column $n$  is the ratio between $\vert E(Q_n(1010))\vert$ and the number of edges of the hypercube having a number of vertices equal to $N=\vert V(Q_n(1010))\vert$, i.e. $(N\log N)/2$ (cf. Proposition \ref{p-vert-edge-hyper}).
Moreover, for each $n$, the value of $R_\sim (11100)$ at column $n$ is the ratio between $\vert E(\tilde Q_n(11100))\vert$ and the number of edges of the tilde-hypercube having a number of vertices equal to $N=\vert V(\tilde Q_n(11100))\vert$, that is $N(3\log N-1)/4$ (cf. Equation \ref{eq-tilde-hyper}).

We guess that both Fibonacci cubes and tilde-Fibonacci cubes 
are the best isometric cubes avoiding a word in terms of reduction of the number of edges, but at the moment the investigation is too germinal. We plan to continue the research in this direction and, above all, to study in deep  structural and topological properties of tilde-Fibonacci cubes.

\bibliographystyle{splncs04}
\bibliography{references}

\end{document}